\documentclass{amsart}
\usepackage{amsfonts}
\usepackage{amsmath,amssymb}
\usepackage{amsthm}
\usepackage{amscd}
\usepackage{graphics}
\usepackage{graphicx}
\theoremstyle{definition}{

\newtheorem{Rem}{{\rm Remark}}

}
\theoremstyle{plain}
{

\newtheorem{Prop}{Proposition}
\newtheorem{Thm}{Theorem}
\newtheorem{MThm}{Main Theorem}

}
\usepackage{pb-diagram}
\usepackage[all]{xy}

\begin{document}
\title[Characterizing some $6$-dimensional manifolds via special generic maps]{Characterizing certain classes of $6$-dimensional closed and simply-connected manifolds via special generic maps}
\author{Naoki Kitazawa}
\address{Institute of Mathematics for Industry, Kyushu University, 744 Motooka, Nishi-ku Fukuoka 819-0395, Japan\\
 TEL (Office): +81-92-802-4402 \\
 FAX (Office): +81-92-802-4405 \\
}
\email{n-kitazawa@imi.kyushu-u.ac.jp}
\urladdr{https://naokikitazawa.github.io/NaokiKitazawa.html}

\begin{abstract}
The present paper finds new necessary and sufficient conditions for $6$-dimensional closed and simply-connected manifolds of certain classes to admit {\it special generic} maps into certain Euclidean spaces.

The class of {\it special generic} maps naturally contains 
Morse functions with exactly two singular points on spheres in so-called Reeb's theorem, characterizing spheres topologically, and canonical projections of unit spheres. Our paper concerns variants of Reeb's theorem. Several results are known e. g. the cases where the manifolds of the targets are the plane and some cases where the manifolds of the domains are closed and simply-connected. Our paper concerns $6$-dimensional versions of a result of Nishioka, determining $5$-dimensional closed and simply-connected manifolds admitting special generic maps into Euclidean spaces completely. Closed and simply-connected manifolds are central geometric objects in (classical) algebraic topology and differential topology. The $6$-dimensional case is more complicated than the $5$-dimensional one: they are classified via explicit algebraic systems.


\end{abstract}


\maketitle
\section{Introduction.}
\label{sec:1}
  
Morse functions on spheres with exactly two singular points are central objects in Reeb's theorem: a closed manifold admits such a function if and only if it is homeomorphic to a $k$-dimensional sphere for $k \neq 4$ or diffeomorphic to the $4$-dimensional unit sphere.
For general theory of Morse functions including these functions, see also \cite{milnor,milnor2}.
  
The class of special generic maps is a certain class of smooth maps whose codimensions are not positive and this class contains these functions and canonical projections of unit spheres as simplest examples.

More rigorously, a smooth map from an $m$-dimensional smooth manifold with no boundary into an $n$-dimensional manifold with no boundary is a special generic map if at each {\it singular} point
it is represented by the form\\
$(x_1, \cdots, x_m) \mapsto (x_1,\cdots,x_{n-1},\sum_{j=1}^{m-n+1} {x_{j+n-1}}^2)$ ($m \geq n \geq 1$)
for suitable local coordinates: a {\it singular} point of a differentiable map means a point in the manifold of the domain where the rank of the differential is smaller than both the dimensions of the manifolds of the domains and the targets.

Note also that special generic maps are so-called {\it fold maps}. Morse functions are also fold maps. Related theory on singularities of differentiable functions and maps are presented systematically in \cite{golubitskyguillemin} for example.

 \cite{burletderham,calabi,furuyaporto} are pioneering studies on special generic maps. Since the 1990s, \cite{saeki} and related studies such as \cite{nishioka,saeki2,sakuma,sakuma2,saekisakuma,saekisakuma2,wrazidlo,wrazidlo2} have discovered interesting phenomena closely related to algebraic topology and differential topology of manifolds.

Reeb's theorem is a kind of characterization theorems of certain classes of (closed) manifolds. The present paper is on variants of this theorem. 

${\mathbb{R}}^k$ denotes the $k$-dimensional Euclidean space (${\mathbb{R}}^1$ is denoted by $\mathbb{R}$ usually). It is a smooth manifold and it admits a natural Riemannian metric: the {\it standard Euclidean} metric. For $x \in {\mathbb{R}}^k$, $||x|| \geq 0$ denotes the distance between $x$ and the origin $0 \in {\mathbb{R}}^k$ where the underlying metric is the standard Euclidean metric. $S^k:=\{x \in {\mathbb{R}}^{k+1}\mid ||x||=1\}$ is the $k$-dimensional unit sphere and a $k$-dimensional compact smooth closed submanifold with no boundary in ${\mathbb{R}}^{k+1}$. $D^k:=\{x \in {\mathbb{R}}^{k}\mid ||x|| \leq 1\}$ is the $k$-dimensional unit disk and a $k$-dimensional connected and compact smooth submanifold in ${\mathbb{R}}^{k}$. A {\it homotopy sphere} means a smooth manifold which is homeomorphic to a (unit) sphere. If it is diffeomorphic to a unit sphere, then it is said to be a {\it standard} sphere. Known smooth manifolds homeomorphic to unit disks are in fact diffeomorphic to them. 

Here a {\it smooth} bundle means a bundle whose fiber is a smooth manifold and whose structure group consists of smooth diffeomorphisms.
A {\it linear} bundle means a smooth bundle whose fiber is a Euclidean space, unit disk, or a unit sphere and whose structure group consists of (natural) linear transformations. For general theory of linear bundles and general bundles, see \cite{milnorstasheff,steenrod} for example.
Connected sums and boundary connected sums of manifolds are considered in the smooth category throughout the present paper.

We introduce some known results and Main Theorems.
\begin{Thm}
	\label{thm:1}
	Let $m \geq n \geq 1$ be integers.
	\begin{enumerate}
		\item {\rm (\cite{saeki})}
		\label{thm:1.1}
		Let $m \geq 2$. An $m$-dimensional closed and connected manifold admits a special generic map into ${\mathbb{R}}^2$ if and only if either of the following two holds.
		\begin{enumerate}
			\item $M$ is a homotopy sphere where $m \neq 4$ or a standard sphere where $m=4$.
			\item $M$ is diffeomorphic to a manifold represented as a connected sum of the total spaces of smooth bundles over $S^1$ whose fibers are either of the following two.
			\begin{enumerate}
				\item An {\rm (}$m-1${\rm )}-dimensional homotopy sphere where $m \neq 5$.
				\item A $4$-dimensional standard sphere where $m=5$.
				\end{enumerate}
		\end{enumerate}
		\item {\rm (\cite{saeki})}
		\label{thm:1.2}
		Let $m=4,5,6$. An $m$-dimensional closed and simply-connected manifold admits a special generic map into ${\mathbb{R}}^3$ if and only if either of the following two holds.
		\begin{enumerate}
			\item $M$ is an $m$-dimensional standard sphere.
			\item $M$ is diffeomorphic to a manifold represented as a connected sum of the total spaces of linear bundles over $S^2$ whose fibers are diffeomorphic to the {\rm (}$m-2${\rm )}-dimensional unit sphere.
		\end{enumerate}
		\item {\rm (\cite{nishioka})}
		\label{thm:1.3}
		Let $m=5$. An $m$-dimensional closed and simply-connected manifold admits a special generic map into ${\mathbb{R}}^4$ if and only if either of the following two holds.
		\begin{enumerate}
			\item $M$ is a $5$-dimensional standard sphere.
			\item $M$ is diffeomorphic to a manifold represented as a connected sum of the total spaces of linear bundles over $S^2$ whose fibers are diffeomorphic to the $3$-dimensional unit sphere.
		\end{enumerate}
	\end{enumerate}
\end{Thm}

\begin{MThm}
	\label{mthm:1}
	A $6$-dimensional closed and simply-connected manifold admits a special generic map into ${\mathbb{R}}^4$ if and only if $M$ is either of the following two.
	\begin{enumerate}
		\item A $6$-dimensional standard sphere.
						\item A manifold diffeomorphic to one represented as a connected sum of finitely many copies of the following manifolds.
						\begin{enumerate}
							\item $S^3 \times S^3$.
							\item The total space of a linear bundle over $S^2$ whose fiber is diffeomorphic to the $4$-dimensional unit sphere.
						\end{enumerate}
						
					\end{enumerate}

	\end{MThm}
For our further new results, we also need fundamental notions on algebraic topology such as homology groups, cohomology groups and cohomology rings. The fundamental terminologies and notions on algebraic topology will be reviewed in section \ref{sec:2} where we regard readers have related knowledge to some extent.

$\mathbb{Z} \subset \mathbb{R}$ denotes the ring of all integers. The notation on homology groups and cohomology groups is presented again later. The {\it singular set} of a differentiable map is the set of all singular points of the map. Propositions \ref{prop:1}, \ref{prop:2} and \ref{prop:3} explain about fundamental properties including properties on the singular sets of special generic maps.

	As Theorem \ref{thm:2} shows, a special generic map $f:M \rightarrow {\mathbb{R}}^5$ on a $6$-dimensional closed and simply-connected manifold $M$ is represented as the composition of a smooth submersion $q_f$ onto a $5$-dimensional compact and simply-connected smooth manifold $W_f$ with a smooth immersion $\bar{f}:W_f \rightarrow {\mathbb{R}}^5$. Furthermore, as Proposition \ref{prop:2} (Proposition \ref{prop:3}) shows, the restriction of the map to the singular set is a diffeomorphism onto the boundary $\partial W_f$.
\begin{MThm}
	\label{mthm:2}	
	A $6$-dimensional closed and simply-connected manifold $M$ admits a special generic map $f:M \rightarrow {\mathbb{R}}^5$ such that the 2nd homology group $H_2(W_f;\mathbb{Z})$ of the $5$-dimensional compact and simply-connected smooth manifold $W_f$ just before is trivial if and if $M$ is either of the following two.
	\begin{enumerate}
		\item A $6$-dimensional standard sphere.
		\item A manifold diffeomorphic to one represented as a connected sum of finitely many copies of the following manifolds.
		\begin{enumerate}
			\item $S^3 \times S^3$.
			\item $S^2 \times S^4$.
		\end{enumerate}
		
	\end{enumerate}

\end{MThm}

We also need some notions on characteristic classes and obstruction theory for linear bundles.
For {\it j-th Stiefel-Whitney classes} and {\it Pontrjagin classes} of (real) vector bundles, linear bundles, tangent bundles, normal bundles of submanifolds and smooth manifolds, for example, see \cite{milnorstasheff} as one of related well-known books. This presents related systematic expositions. $\mathbb{Z}/2\mathbb{Z}$ is the field of order $2$.
\begin{MThm}
	\label{mthm:3}
A $6$-dimensional, closed and simply-connected manifold $M$ such that the $2$nd homology group $H_2(M;\mathbb{Z})$ is finite and has no elements which are not the identity elements and whose orders are finite and even admits a special generic map into ${\mathbb{R}}^5$ whose singular set is connected if and only if the following three hold.
\begin{enumerate}
	\item \label{mthm:3.1} The 2nd Stiefel Whitney class of $M$, which is the uniquely defined element of the cohomology group $H^2(M;\mathbb{Z}/2\mathbb{Z})$, is the zero element.
	\item \label{mthm:3.2} The 1st Pontrjagin class of $M$, which is the the uniquely defined element of the cohomology group $H^4(M;\mathbb{Z})$ for an arbitrary oriented manifold $M$, is the zero element.
	\item \label{mthm:3.3} The cup product $c_1 \cup c_2$ is the zero element for any pair $c_1,c_2 \in H^2(M;A)$ of cohomology classes in the cohomology group $H^2(M;A)$ for $A:=\mathbb{Z}$ and any finite field $A$.	
\end{enumerate}

\end{MThm}
These results such as Main Theorems \ref{mthm:1}, \ref{mthm:2} and \ref{mthm:3} can be regarded as $6$-dimensional variants of Theorem \ref{thm:1} (\ref{thm:1.3}), for example. Note that for example, in some of Theorem \ref{thm:1}, explicit classifications of closed and simply-connected manifolds of certain classes such as \cite{barden,wall} are key ingredients.
 In our new study, \cite{jupp,wall,zhubr,zhubr2} are key results on classifications of such manifolds.


We have another main theorem. 

\begin{MThm}
\label{mthm:4}
\begin{enumerate}
	\item 
	\label{mthm:4.1}
	Assume that a closed and simply-connected manifold $M$ of dimension $m \geq 5$ admits a special generic map into ${\mathbb{R}}^4$, then $M$ enjoys the following two properties.
\begin{enumerate}
	\item \label{thm:4.1.1} The $j$-th Steifel-Whitney class of $M$, which is the uniquely defined element of $H^j(M;\mathbb{Z}/2\mathbb{Z})$, is the zero element for any integer $j \geq 0$ except $j=2$.
	\item \label{thm:4.1.2} The $j$-th Pontrjagin class of $M$, which is the uniquely defined element of $H^{4j}(M;\mathbb{Z})$, is the zero element for any integer $j \geq 0$ and an arbitrary oriented $M$.
	 \end{enumerate}
 \item \label{mthm:4.2}
 Let $M^{7,0}$ be a $7$-dimensional closed and simply-connected manifold diffeomorphic to a manifold represented as a connected sum of finitely many manifolds in the following two. 
 \begin{itemize}
 	\item The total space of a linear bundle over $S^2$ whose fiber is the $5$-dimensional unit sphere.
 	\item A copy of $S^3 \times S^4$.
 \end{itemize}
Furthermore, the family of the finitely many manifolds contains at least one copy of $S^3 \times S^4$. Then there exists a family $\{M^{7,\lambda}\}_{\lambda \in \Lambda}$ of countably many $7$-dimensional closed and simply-connected manifolds satisfying the following two.
\begin{enumerate}
	\item \label{mthm:4.2.1} $M^{7,{\lambda}_1}$ and $M^{7,{\lambda}_2}$ are not homeomorphic for distinct elements ${\lambda}_1, {\lambda}_2 \in \Lambda$.
	\item \label{mthm:4.2.2} There exists an isomorphism ${\phi}_{\lambda}$ from the cohomology ring $H^{\ast}(M^{7,0};\mathbb{Z})$ of $M^{7,0}$ onto the cohomology ring $H^{\ast}(M^{7,\lambda};\mathbb{Z})$ of $M^{7,\lambda}$.
	\item \label{mthm:4.2.3} For ${\phi}_{\lambda}$ before, by considering the natural quotient map from $\mathbb{Z}$ to $\mathbb{Z}/2\mathbb{Z}$, we canonically obtain an isomorphism ${\phi}_{\lambda,2}$ from the cohomology ring $H^{\ast}(M^{7,0};\mathbb{Z}/2\mathbb{Z})$ onto the cohomology ring $H^{\ast}(M^{7,\lambda};\mathbb{Z}/2\mathbb{Z})$. This maps the $j$-th Stiefel Whitney class of $M^{7,0}$ to that of $M^{7,\lambda}${\rm :} the $j$-th Stiefel-Whitney classes are defined uniquely as the elements of the cohomology groups as before, of course. 
	\item \label{mthm:4.2.4} $M^{7,\lambda}$ does not admit special generic maps into ${\mathbb{R}}^4$, whereas $M^{7,0}$ admits ones. $M^{7,\lambda}$ and $M^{7,0}$ admit special generic maps into ${\mathbb{R}}^5$.
\end{enumerate} 
\end{enumerate} 
\end{MThm}

This is closely related to Main Theorem \ref{mthm:1}. In addition, this extends some results of section 3 of \cite{saeki}.

Our new results and related facts explicitly show that special generic maps are attractive in algebraic topology and differential topology of manifolds although the class of special generic maps seems to be not so wide considering the definition. 

In the next section we review fundamental properties of special generic maps. The third section is devoted to Main Theorems and Theorems \ref{thm:3}, \ref{thm:4} and \ref{thm:5}.
The fourth section is for concluding remarks.

\section{Fundamental properties of special generic maps.}
\label{sec:2}
\begin{Prop}
\label{prop:1}
Let $f$ be a special generic map from an $m$-dimensional manifold with no boundary into $n$-dimensional manifold with no boundary where $m \geq n$. Then the following properties hold.
\begin{enumerate}
\item \label{prop:1.1}
 The {\rm singular set} of $f$, defined as the set of all singular points of $f$, is an {\rm (}$n-1${\rm )}-dimensional smooth closed submanifold with no boundary. Furthermore, the restriction of $f$ there is a smooth immersion.
\item \label{prop:1.2}
 $f$ is, for suitable local coordinates, represented as the product map of a Morse function and the identity map on a small open neighborhood of each singular point where the open neighborhood is taken in the singular set and of dimension $n-1$.
\end{enumerate}
\end{Prop}

\begin{Prop}[E. g. \cite{saeki}]
\label{prop:2}
Let $m \geq n \geq 1$ be integers.
For a special generic map $f:M \rightarrow N$ on an $m$-dimensional closed manifold $M$ into an $n$-dimensional manifold $N$ with no boundary, the following properties hold.
\begin{enumerate}
\item
\label{prop:2.1}
 There exists some $n$-dimensional compact smooth manifold $W_f$ and some smooth immersion $\bar{f}:W_f \rightarrow N$. Furthermore, for example, $W_f$ can be taken as follows.
 \begin{enumerate}
 	\item If the manifold $N$ of the target is orientable, then $W_f$ is taken as an orientable manifold. 
 	\item If the manifold $M$ of the domain is connected, then $W_f$ is taken as a connected manifold. 
 	\item If the manifold $N$ of the target is orientable and the manifold $M$ of the domain is connected, then $W_f$ is taken as a connected and orientable manifold.
 	\end{enumerate}
\item
\label{prop:2.2}
 We have a smooth surjection $q_f:M \rightarrow W_f$ with the relation $f=\bar{f} \circ q_f$.
\item
\label{prop:2.3}
 $q_f$ maps the singular set of $f$ onto the boundary $\partial W_f \subset W_f$. The restriction of $q_f$ there is also regarded as a diffeomorphism onto $\partial W_f$.
\item
\label{prop:2.4}
\begin{enumerate}
\item
\label{prop:2.4.1}
 For some small collar neighborhood $N(\partial W_f) \subset W_f$, the composition of the map $q_f {\mid}_{{q_f}^{-1}(N(\partial W_f))}$ onto $N(\partial W_f)$ with the canonical projection to $\partial W_f$ can be regarded as the projection of a linear bundle whose fiber is the {\rm(}$m-n+1${\rm )}-dimensional unit disk. Note that this will be in Proposition \ref{prop:3} defined as a {\rm boundary linear bundle} of $f$.
\item
\label{prop:2.4.2}
 The restriction of $q_f$ to the preimage of $W_f-{\rm Int}\ N(\partial W_f)$ can be regarded as the projection of a smooth bundle over $W_f-{\rm Int}\ N(\partial W_f)$ whose fiber is an {\rm (}$m-n${\rm )}-dimensional standard sphere. Note that this will be in Proposition \ref{prop:3} defined as an {\rm internal smooth bundle} of $f$.
\end{enumerate}
\end{enumerate}
\end{Prop}

This is presented for the case $m>n \geq 1$ in \cite{saeki} only. We can easily check for the case $m=n$ and related theory is also discussed in section 8 there.

Conversely, we have the following proposition.
\begin{Prop}[E. g. \cite{saeki}]
	\label{prop:3}
Let $m \geq n \geq 1$ be integers. Let a smooth immersion $\bar{f}:W_f \rightarrow N$ of an $n$-dimensional compact smooth manifold $W_f$ into an $n$-dimensional manifold $N$ with no boundary be given. We also assume the existence of the two bundles in the following first two conditions and the third condition.
 \begin{itemize}
 	\item A linear bundle over $\partial W_f$ whose fiber is the {\rm (}$m-n+1${\rm )}-dimensional unit disk.
 	\item A smooth bundle over $W_f-{\rm Int}\ N(W_f)$ whose fiber is an {\rm (}$m-n${\rm )}-dimensional standard sphere where a suitable small collar neighborhood $N(\partial W_{f}) \subset W_{f}$ is taken. 
 	\item The subbundle of the former bundle whose fiber is $\partial D^{m-n+1}$ and the restriction of the latter bundle to the boundary are equivalent as smooth bundles over $N(\partial W_f)-\partial W_f$. Note that for the former bundle, we consider a natural identification between $\partial W_f$ and $N(\partial W_f)-\partial W_f$ defined naturally from the structure of the collar neighborhood $N(\partial W_f)$ in $W_f$.    
 \end{itemize}

We call the first bundle a {\rm boundary linear bundle over} $W_f$ and the second bundle an {\rm internal smooth bundle over} $W_f$. This respects expositions in \cite{kitazawa6}.

Then we have a special generic map $f:M \rightarrow N$ on a suitable $m$-dimensional closed manifold $M$ into $N$ satisfying the following properties.
\begin{enumerate}
\item There exists a smooth surjection $q_{f}:M \rightarrow W_f$ satisfying the relation $f=\bar{f} \circ q_{f}$.
\item $q_{f}$ maps the singular set of $f$ onto the boundary $\partial W_f \subset W_f$. This is regarded as a diffeomorphism.
\item The composition of the map $q_f {\mid}_{{q_f}^{-1}(N(\partial W_f))}$ onto $N(\partial W_f)$ with the canonical projection to $\partial W_f$ can be regarded as the projection of a linear bundle equivalent to the boundary linear bundle over $W_f$ before. We can canonically define such a bundle as a {\rm boundary linear bundle of $f$}.
\item The restriction of $q_f$ to the preimage of $W_f-{\rm Int}\ N(\partial W_f)$ can be regarded as the projection of a smooth bundle over $W_f-{\rm Int}\ N(\partial W_f)$ equivalent to the internal smooth bundle over $W_f$ before. We can canonically define such a bundle as an {\rm internal smooth bundle of $f$}.
\end{enumerate}
\end{Prop}
Simplest examples are obtained by setting 
the two bundles as trivial ones.

Smooth manifolds have canonical PL structures and regarded as polyhedra canonically. In Proposition \ref{prop:4} and the present paper, this is important.

\begin{Prop}[E. g. \cite{saeki}]
\label{prop:4}
Let $m>n \geq 1$ be integers. Let $f:M \rightarrow N$ be a special generic map on an $m$-dimensional closed and connected manifold $M$ into an $n$-dimensional manifold $N$ with no boundary. Then we have an {\rm (}$m+1${\rm )}-dimensional compact and connected {\rm (PL)} manifold $W$ satisfying the following {\rm (}differential{\rm )} topological properties.
 \begin{enumerate}
 	\item \label{prop:4.1} The boundary of $W$ is $M$ where we consider in the topology or PL category.
 	\item \label{prop:4.2} $W$ collapses to $W_f$ where we abuse the notation in Propositions \ref{prop:2} and \ref{prop:3}.
 	\item \label{prop:4.3} $W_f$ can be also identified with a suitable CW subcomplex {\rm (}resp. subpolyhedron{\rm )} of $W$.
 	\item \label{prop:4.4} Let $i_M:M \rightarrow W$ denote the canonical inclusion. Then for a suitable continuous {\rm (}resp .{\rm PL)} map $r_f:W \rightarrow W_f$ giving a collapsing to $W_f$, we have the relation $q_f=r_f \circ i_M$. Furthermore, $r_f$ can be regarded as the projection of a bundle over $W_f$ whose fiber is diffeomorphic to $D^{m-n+1}$ and which may not be a smooth bundle. Furthermore, we may assume the following two where we can abuse the notions, terminologies and notation in Proposition \ref{prop:3}. 
 	\begin{enumerate}
 		\item Consider the restriction of $r_f$ to the preimage of $W_f-{\rm Int}\ N(\partial W_f)$. The subbundle obtained by restricting the fiber to $\partial D^{m-n+1} \subset D^{m-n+1}$ is equivalent to an internal smooth bundle of $f$.
 		\item Consider the restriction of $r_f$ to the preimage of $N(\partial W_f)$. Next consider the composition of this projection with the canonical projection to $\partial W_f$. This can be regarded as the projection of a bundle whose fiber is diffeomorphic to $D^{m-n+1} \times D^1$.
 		Furthermore, consider the subbundle whose fiber is $(\partial D^{m-n+1} \times D^1) \cup (D^{m-n+1} \times {S^0}_0)$ where ${S^0}_0 \subset S^0=\partial D^1$ denotes a one-point set. It is equivalent to a boundary linear bundle of $f$.
 	
  \end{enumerate}
 
 	\item \label{prop:4.5} In the case $m-n=1,2,3$ for example, $W$ and $r_f$ can be chosen as a smooth manifold and a smooth map and the bundle over $W_f$ can be regarded as a smooth bundle with the projection $r_f$.
    \item \label{prop:4.6}
     In the case where $M$ and $N$ are oriented, we can take $W$ as an oriented manifold whose boundary is $M$. As before, in the case $m-n=1,2,3$ for example, $W$ and $r_f$ can be chosen as a smooth manifold and a smooth map and the bundle over $W_f$ can be regarded as a smooth bundle with the projection $r_f$.
    \end{enumerate}
\end{Prop}

For more general propositions, see \cite{saekisuzuoka} and see papers \cite{kitazawa0.1,kitazawa0.2,kitazawa0.3} and several preprints in References by the author.

We introduce notation and terminologies on {\it homology groups}, {\it cohomology groups}, {\it cohomology rings} and {\it homotopy groups}. For systematic expositions on such notions and fundamental algebraic topology, see \cite{hatcher} for example.

Let $(X,X^{\prime})$ be a pair of topological spaces satisfying $X^{\prime} \subset X$. We allow $X^{\prime}$ to be the empty set. 
Let $A$ be a commutative ring. 
$A$ is, for example, taken as the ring of all integers, denoted by $\mathbb{Z} \subset \mathbb{R}$. 

The $j$-th {\it homology group} and {\it cohomology group} of the pair $(X,X^{\prime})$ of topological spaces satisfying $X^{\prime} \subset X$ are defined and denoted by $H_{j}(X,X^{\prime};A)$ and $H^{j}(X,X^{\prime};A)$ where $A$ is the {\it coefficient ring}. If $X^{\prime}$ is empty, then we may omit ",$X^{\prime}$" in the notation and the homology group and the cohomology group of the pair $(X,X^{\prime})$ are also called the {\it homology group} and the {\it cohomology group} of $X$. {\rm (}{\it Co}{\rm )}{\it homology classes} of $(X,X^{\prime})$ or $X$ mean elements of the (resp. co)homology groups. 

The {\it $k$-th homotopy group} of a topological space $X$ is denoted by ${\pi}_k(X)$.

Let $(X_1,{X_1}^{\prime})$ and $(X_2,{X_2}^{\prime})$ be pairs of topological spaces satisfying ${X_1}^{\prime} \subset X_1$ and ${X_2}^{\prime} \subset X_2$ where the second topological spaces of these pairs are allowed to be the empty sets as before. For a continuous map $c:X_1 \rightarrow X_2$ satisfying $c({X_1}^{\prime}) \subset {X_2}^{\prime}$, $c_{\ast}:H_{\ast}(X_1,{X_1}^{\prime};A) \rightarrow H_{\ast}({X_2},{X_2}^{\prime};A)$ and $c^{\ast}:H^{\ast}({X_2},{X_2}^{\prime};A) \rightarrow H^{\ast}(X_1,{X_1}^{\prime};A)$ denote the natural homomorphisms defined canonically. For a continuous map $c:X_1 \rightarrow X_2$, $c_{\ast}:{\pi}_k(X_1) \rightarrow {\pi}_k(X_2)$ also denotes the natural homomorphism between the homotopy groups of degree $k$ which is also defined canonically. 

Let $H^{\ast}(X;A)$ denote the direct sum of the $j$-th cohomology groups ${\oplus}_{j=0}^{\infty} H^j(X;A)$ for every integer $j \geq 0$. The cup products for a pair $c_1,c_2 \in H^{\ast}(X;A)$ and a sequence $\{c_j\}_{j=1}^l \subset H^{\ast}(X;A)$ of $l>0$ cohomology classes are important. $c_1 \cup c_2$ and ${\cup}_{j=1}^l c_j$ denote them. The former is  regarded as a specific case or the case for $l=2$.
This makes $H^{\ast}(X;A)$ a graded commutative algebra and we call this the {\it cohomology ring} of $X$ whose {\it coefficient ring} is $A$.  

\section{The Proofs of our Main Theorems.}
\label{sec:3}
\begin{Thm}
\label{thm:2}
Let $m>n \geq 1$ be integers. Let $l>0$ be another integer.
Let $M$ be an $m$-dimensional closed and connected manifold.
Let $A$ be a commutative ring. 
\begin{enumerate}
	\item \label{thm:2.1} {\rm (E. g. \cite{saeki})}
	 Let $f:M \rightarrow N$ be a special generic map into an $n$-dimensional connected and non-closed manifold with no boundary.
	Then the homomorphisms ${q_f}_{\ast}:H_j(M;A) \rightarrow H_j(W_f;A)$, ${q_f}^{\ast}:H^j(W_f;A) \rightarrow H^j(M;A)$ and  ${q_f}_{\ast}:{\pi}_j(M) \rightarrow {\pi}_j(W_f)$ are also isomorphisms for $0 \leq j \leq m-n$ where we abuse the notation in Propositions \ref{prop:2} and \ref{prop:3}.
	\item \label{thm:2.2} {\rm (\cite{kitazawa})} Let there exist a sequence $\{a_j\}_{j=1}^l \subset H^{\ast}(M;A)$ satisfying the following two. 
\begin{itemize}
\item The cup product ${\cup}_{j=1}^l a_j$ is not zero.
\item The degree of each cohomology class in $\{a_j\}_{j=1}^l$ is smaller than or equal to $m-n$. The sum of the degrees for all these $l$ cohomology classes is greater than or equal to $n$.
\end{itemize}
Then $M$ does not admit special generic maps into any $n$-dimensional connected and non-closed manifold which has no boundary.
\end{enumerate}
\end{Thm}
We review a proof omitting expositions on a {\it handle} and its {\it index} for PL manifolds and general polyhedra.
\begin{proof}
We first show (\ref{thm:2.1}).
We can take an ($m+1$)-dimensional compact and connected PL manifold $W$ in Proposition \ref{prop:4}. 
$W_f$ is a compact smooth manifold and smoothly immersed into $N$. It is simple homotopy equivalent to an ($n-1$)-dimensional compact and connected polyhedron. 
$W$ is simple homotopy equivalent to $W_f$. $W$ is shown to be a PL manifold obtained by attaching handles to $M \times \{0\} \subset M \times [-1,0]$ whose indices are greater than $(m+1)-{\dim W_f}=m-n+1$. Here $M$ can be and is identified with $M \times \{-1\} \subset M \times [-1,0]$.

From this we have (\ref{thm:2.1}).

We show (\ref{thm:2.2}). Suppose that $M$ admits a special generic map into an $n$-dimensional connected and non-closed manifold $N$ with no boundary. We have an ($m+1$)-dimensional manifold $W$ as just before. By virtue of (\ref{thm:1.1}) and Proposition \ref{prop:4} (\ref{prop:4.4}), we can take a unique cohomology class $b_j \in H^{\ast}(W;A)$ satisfying $a_j={i_M}^{\ast}(b_j)$ where $i_M$ is as in Proposition \ref{prop:4}. $W$ has the simple homotopy type of an ($n-1$)-dimensional polyhedron. This means that the cup product ${\cup}_{j=1}^l a_j$ is zero, which contradicts the assumption. This completes the proof.
\end{proof}

We shortly review several notions on homology classes of manifolds. 
$A$ is a commutative ring having the identity element different from the zero element.
The {\it fundamental class} of a compact, connected and oriented smooth, PL, or more generally, topological manifold $Y$ is the canonically defined ($\dim Y$)-th homology class. This is the generator of $H_{\dim Y}(Y,\partial Y;A)$, which is isomorphic to $A$, and canonically and uniquely defined from the orientation. 
Let $i_{Y,X}:Y \rightarrow X$ be an embedding satisfying $i_{Y,X}(\partial Y) \subset \partial X$ and $i_{Y,X}({\rm Int}\ Y) \subset {\rm Int}\ X$ where we consider the embedding as a suitable one in the suitable category.

Let $h \in H_{j}(X,\partial X;A)$. If the value of the homomorphism ${i_{Y,X}}_{\ast}$ induced by the embedding $i_{Y,X}:Y \rightarrow X$ at the fundamental class of $Y$ is $h$, then $h$ is said to be {\it represented} by the oriented submanifold $Y$. 

We explain about the {\it Poincar\'e dual} and the {\it dual} to each element of a basis of a module without explaining about rigorous definitions. In various scenes of the present paper, we consider Poincar\'e duality theorem for a compact and connected manifold and we consider the {\it Poincar\'e dual} to a homology class or a cohomology class. Consider a basis of a module whose elements are not divisible by elements which are not unit elements. Each of the elements of the basis, we have the {\it dual} to it and this is also important. For a homology group or its subgroup and its basis as before, we naturally have a cohomology class uniquely as the dual to a homology class of the basis. Their degrees are same. This notion is different from the notion of a Poincar\'e dual.

For notions here, see \cite{hatcher} again for example.

\begin{proof}[A proof of Main Theorem \ref{mthm:1}]
	We abuse the notation in Propositions \ref{prop:2} and \ref{prop:3} and Theorem \ref{thm:2} for example.
	
	Suppose that a $6$-dimensional closed and simply-connected manifold $M$ admits a special generic map $f:M \rightarrow {\mathbb{R}}^4$.
	According to Theorem \ref{thm:2} (\ref{thm:2.1}), $M$ and $W_f$ are simply-connected and $H_2(M;\mathbb{Z})$ is isomorphic to $H_2(W_f;\mathbb{Z})$.
	According to \cite{nishioka}, $H_2(W_f;\mathbb{Z})$ is free and $H_2(M;\mathbb{Z})$ is free as a result. Note that this is a key ingredient in showing Theorem \ref{thm:1} (\ref{thm:1.3}). We can take a basis of $H_2(W_f;\mathbb{Z})$ and smoothly and disjointly embedded $r \geq 0$ $2$-dimensional spheres in $W_f-{\rm Int}\ N(\partial W_f)$ enjoying the following two.
	\begin{itemize}
		\item $r$ denotes the ranks of $H_2(M;\mathbb{Z})$ and $H_2(W_f;\mathbb{Z})$.
		\item Each of these $r$ elements of the basis is represented by each of these $r$ $2$-dimensional spheres in $W_f-{\rm Int}\ N(\partial W_f)$.
	\end{itemize}  
Furthermore, the restriction of the bundle over $W_f-{\rm Int} N(W_f)$ in Proposition \ref{prop:3} to each $2$-dimensional sphere of the $r \geq 0$ spheres admits a section due to some general theory of linear bundles. This bundle over $W_f-{\rm Int} N(W_f)$ is a linear bundle whose fiber is diffeomorphic to $S^2$ and an internal smooth bundle of $f$. By Poincar\'e duality theorem, we have a basis of $H_2(W_f,\partial W_f;\mathbb{Z})$ each of which is the Poincar\'e dual to the dual to the each element of the basis of $H_2(W_f;\mathbb{Z})$. 
	
	Homology classes in $H_2(W_f,\partial W_f;\mathbb{Z})$ are represented by compact, connected and oriented surfaces. This is a very specific case of Thom's theory \cite{thom} on homology classes of compact manifolds represented by oriented submanifolds.
	
	We thus have $r \geq 0$ compact, connected and oriented surfaces smoothly embedded in $W_f$ such that the $r$ Poincar\'e duals before are represented by them. Note also that their interiors are embedded in the interior ${\rm Int}\ W_f$ and that their boundaries are embedded in the boundary $\partial W_f$.
	Furthermore, by the definition of a special generic map, Proposition \ref{prop:1} (Proposition \ref{prop:1} (\ref{prop:1.2})) and fundamental theory of so-called {\it generic} smooth embeddings and smooth maps, the preimages of the $r$ compact, connected and oriented surfaces are regarded as the $4$-dimensional closed, connected and orientable manifolds of the domains of some special generic maps into orientable surfaces with no boundaries. Note that the orientability of these $4$-dimensional manifolds follows from the fact that (the total space of) an internal smooth bundle of each special generic map $f_{\lambda}$ is orientable for example.
	
	 We have exactly $r$ homology classes in $H_4(M;\mathbb{Z})$ each of which is represented by each of these $r$ $4$-dimensional closed, connected and suitably oriented manifolds. They can be regarded as the Poincar\'e duals to the duals to elements of a suitable basis of $H_2(M;\mathbb{Z})$. The set of all of these homology classes can be regarded as a basis of $H_4(M;\mathbb{Z})$.
	 For example, we can choose a basis of $H_2(M;\mathbb{Z})$ each element of which is represented by (the image of) the section of each bundle over the $r$ bundles over the $r$ $2$-dimensional spheres in ${\rm Int}\ W_f$ before.
	 
	 We can see that these $r$ $4$-dimensional closed and oriented manifolds which are also smooth closed submanifolds in $M$ are oriented null-cobordant. This follows from Proposition \ref{prop:4} (\ref{prop:4.6}). Note that we apply this to the special generic map $f_{\lambda}$ on the $4$-dimensional manifold. 
	 
	 The special generic map $f_{\lambda}$ is regarded as a map into a non-closed, connected and orientable surface. By Propositions \ref{prop:2} and \ref{prop:3} and the fundamental fact that compact and orientable surfaces can be smoothly immersed into ${\mathbb{R}}^2$, we have a special generic map on the $4$-dimensional manifold into ${\mathbb{R}}^2$. Theorem \ref{thm:1} (\ref{thm:1}) implies that the 2nd cohomology group of the $4$-dimensional manifold is zero where the coefficient ring is $\mathbb{Z}$. This means that a normal bundle of the $4$-dimensional manifold which is a smooth submanifold in $M$ is trivial by a fundamental argument on linear bundles.
From another fundamental argument on linear bundles and characteristic classes again, we have that the 1st Pontrjagin class of $M$ with an arbitrary orientation is zero. 
	 
	 Theorem \ref{thm:2} yields that the cup products of 2nd homology classes for $M$ is always zero where the coefficient ring is $\mathbb{Z}$.
	 
	 These two together with classification theorems of $6$-dimensional closed and simply-connected manifolds such as \cite{wall,zhubr,zhubr2} imply that $M$ is either of the following two.
	 \begin{itemize}
	 	\item A $6$-dimensional standard sphere. 
	 	\item A closed manifold diffeomorphic to one represented as a connected sum of the following two manifolds.
	 	\begin{itemize}
	 		\item (A copy of) $S^3 \times S^3$.
	 		\item The total space of a linear bundle over $S^2$ whose fiber is diffeomorphic to the $4$-dimensional unit sphere.
	 	\end{itemize}
	 \end{itemize}
 
 Every standard sphere admits a special generic map into any Euclidean space whose dimension is at most the dimension of the sphere. It is sufficient to consider canonical projections.
 
 We can easily have a copy of $S^2 \times D^2$ smoothly embedded in ${\mathbb{R}}^4$. By applying Proposition \ref{prop:3}, we have a special generic map on an arbitrary manifold diffeomorphic to the total space of a linear bundle over $S^2$ whose fiber is diffeomorphic to the $4$-dimensional unit sphere. 
 
 It is important that the structure group of a linear bundle over $S^2$ whose fiber is the $k$-dimensional unit sphere $S^k$ ($k \geq 2$) is regarded as the group consisting of linear transformations on $S^1 \subset S^k$, or the $2$-dimensional rotation group $SO(2)$, isomorphic to $S^1$ as a Lie group. Moreover, after inductively choosing equators starting from $S^k$, we have $S^1 \subset S^k$ and the action by the linear transformations are naturally regarded as linear transformations on $S^k$. This is an elementary argument on linear bundles and for related studies, see \cite{milnorstasheff,steenrod} for example. 

 In addition, our construction of a special generic map here is also essentially regarded as construction Nishioka has also shown in \cite{nishioka} in the case where the manifolds are $5$-dimensional.
 
 We consider the product map of a Morse function in Reeb's theorem on $S^3$ and the identity map on $S^3$ and we can smoothly embed the manifold of the target into ${\mathbb{R}}^4$. Thus $S^3 \times S^3$ admits a special generic map into ${\mathbb{R}}^4$.
 
 According to \cite{saeki} for example, construction of special generic maps on manifolds represented as connected sums of manifolds admitting special generic maps into a fixed Euclidean space is easy. Of course we construct ones into the fixed Euclidean space. 
 
 This completes the proof.
\end{proof}
\begin{Thm}
	\label{thm:3} Suppose that a $6$-dimensional, closed and simply-connected manifold $M$ admits a special generic map $f:M \rightarrow {\mathbb{R}}^5$. Let $e_{\rm F}$ be an arbitrary element of the homology group $H_2(M;\mathbb{Z}/2\mathbb{Z})$ which is in the image of the canonically obtained homomorphism ${\phi}_{G,\mathbb{Z},\mathbb{Z}/2\mathbb{Z}}$ associated with some internal direct sum decomposition of the form $G_{\rm Free} \oplus G_{\rm Finite}$ of the homology group $H_2(M;\mathbb{Z})$ where $G_{\rm Free}$ and $G_{\rm Finite}$ are free and finite, respectively. Assume also the following conditions.
	\begin{enumerate}
		\item The homomorphism ${\phi}_{G,\mathbb{Z},\mathbb{Z}/2\mathbb{Z}}$ is defined on the summand $G_{\rm Finite}$.
		\item The homomorphism ${\phi}_{G,\mathbb{Z},\mathbb{Z}/2\mathbb{Z}}$ is the restriction of the homomorphism from $H_2(M;\mathbb{Z})$ into $H_2(M;\mathbb{Z}/2\mathbb{Z})$ defined canonically from the natural quotient map from $\mathbb{Z}$ onto $\mathbb{Z}/2\mathbb{Z}$.
		\item ${q_f}_{\ast}(e_{\rm F})$ is not the zero element where we abuse the notation $q_f$ as before.
	\end{enumerate}
	Then the value of the 2nd Stiefel-Whitney class of $M$, which is the uniquely defined element of the cohomology group $H^2(M;\mathbb{Z}/2\mathbb{Z})$, must be the zero element $0 \in \mathbb{Z}/2\mathbb{Z}$ at some element ${e_{\rm F}}^{\prime} \in H_2(M;\mathbb{Z}/2\mathbb{Z})$ satisfying ${q_f}_{\ast}({e_{\rm F}}^{\prime})={q_f}_{\ast}(e_{\rm F})$.
\end{Thm}
\begin{proof}
	We abuse the notation $q_f$ and $W_f$ as before.
	$M$ and $W_f$ are simply-connected from the assumption and Theorem \ref{thm:2} (\ref{thm:2.1}). This means that $e_{\rm F} \in H_2(M;\mathbb{Z}/2\mathbb{Z})$ is represented by a smoothly embedded $2$-dimensional sphere in $M$ and that $e_{W_f,{\rm F}}:={q_f}_{\ast}(e_{\rm F}) \in H_2(W_f;\mathbb{Z}/2\mathbb{Z})$ is represented by a smoothly embedded $2$-dimensional sphere in ${\rm Int}\ W_f$. If we restrict an internal smooth bundle of $f$ to the smoothly embedded $2$-dimensional sphere in ${\rm Int}\ W_f$, then it must be trivial. Note that $q_f$ can be regarded as the projection of this bundle over the $2$-dimensional sphere (after the restriction). This is due to the assumption that $W_f$ and $M$ can be oriented and a well-known classification theorem of smooth (linear) bundles whose fibers are circles. In short, smooth linear bundles over a fixed space whose fibers are circles and whose structure groups are reduced to ones consisting of orientation preserving liner transformations, or the 2-dimensional rotation group, isomorphic to $S^1$ as a Lie group, are classified by the 2nd cohomology group of the base space whose coefficient ring is $\mathbb{Z}$.
	$W_f$ is a $5$-dimensional compact manifold smoothly immersed into ${\mathbb{R}}^5$ and this means that the tangent bundle is a trivial linear bundle. $e_{W_f,{\rm F}}$ is not the zero element. 
	By virtue of fundamental properties of Stiefel-Whitney classes, presented systematically in \cite{milnorstasheff} for example, and the tangent bundles of several manifolds such as $W_f$, $M$, and the $2$-dimensional sphere $S^2$, this completes the proof.
\end{proof}
\begin{Thm}
	\label{thm:4}
		Assume that a $6$-dimensional closed and simply-connected manifold $M$ admits a special generic map $f:M \rightarrow {\mathbb{R}}^5$ such that the 2nd homology group $H_2(W_f;\mathbb{Z})$ of $W_f$ is finite where we abuse $W_f$ as before. Then we have the following three.
		\begin{enumerate}
			\item \label{thm:4.1}
			 The 2nd Stiefel-Whitney class of $M$, which is the uniquely determined element of $H^2(M;\mathbb{Z}/2\mathbb{Z})$, is the zero element.
			\item \label{thm:4.2}
			The 1st Pontrjagin class of $M$, which is the uniquely determined element of $H^4(M;\mathbb{Z})$ where $M$ is oriented in an arbitrary way, is the zero element.
			\item \label{thm:4.3}
			The cup product of any two elements of $H^2(M;\mathbb{Z})$ is the zero element.
		\end{enumerate}
	
\end{Thm}
\begin{proof}
We abuse the notation in Propositions and
(Main) Theorems in the present paper. 
For example we also abuse "$W$" in Proposition \ref{prop:4} here.

We show (\ref{thm:4.1}) and (\ref{thm:4.2}). $H_2(W_f;\mathbb{Z})$ is finite. The classification theory of circle bundles presented in the proof of Theorem \ref{thm:3}
implies that an internal smooth bundle of $f$ and a boundary linear bundle of $f$ are regarded as trivial bundles here. $W_f$ is, as presented in the proof of Theorem \ref{thm:3}, has the trivial tangent bundle.

By fundamental arguments on characteristic classes, we have (\ref{thm:4.1}) and (\ref{thm:4.2}).

Proposition 3.10 of \cite{saeki} shows a useful homology exact sequence and we use another homology exact sequence. Each "$\cong$" between two groups means they are isomorphic. $0$ denotes the zero element of course. ${\partial}_{\ast}$ denotes the boundary homomorphism for an exact sequence. 

$M$ and $W_f$ are simply-connected from the assumption and Theorem \ref{thm:2} (\ref{thm:2.1}).

Let $A$ be an arbitrary commutative ring.
The exact sequence in \cite{saeki} is

$$\xymatrix{ H_3(W;A) \ar[r]& H_3(W,M;A) \cong H^4(W;A) \cong H^4(W_f;A) \ar[r]^{{\partial}_{\ast}}& \\
	H_2(M;A) \ar[r]^{{i_M}_{\ast}} & H_2(W;A) \cong H_2(W_f;A) \ar[r] & H_2(W,M;A) \cong H^5(W;A) \cong H^5(W_f;A) \cong \{0\} \ar[r]&}
$$
and we also have another homology exact sequence	
$$\xymatrix{ H_2(\partial W_f;A) \ar[r]& H_2(W_f,\partial W_f;A) \ar[r]^{{\partial}_{\ast}}& \\
	H_1(\partial W_f;A) \ar[r] & H_1(W_f;A) \cong \{0\} \ar[r] & H_1(W_f,\partial W_f;A) \cong H^4(W_f;A) \ar[r]& \\
	H_0(\partial W_f;A) \cong A^l:={\oplus}_{j=1}^l A \ar[r]& H_0(W_f;A) \cong A \ar[r]& \{0\}&}
$$
where we apply Poincar\'e duality theorem for compact and simply-connected manifolds $W$ and $W_f$ and Proposition \ref{prop:4} (\ref{prop:4.2}) for several isomorphisms of groups (modules over $A$) for example. $A^l:={\oplus}_{j=1}^{l} A$ denotes the direct sum of exactly $l>0$ copies of $A$ where $l$ is the number of connected components of the singular set of $f$, diffeomorphic to $\partial W_f$ by Proposition \ref{prop:2} (\ref{prop:2.3}).
Let $A:=\mathbb{Z}$. $H_2(W;\mathbb{Z})$ is finite. The rank of $H_2(M;\mathbb{Z})$ is at most $l-1$ since $H^4(W_f;\mathbb{Z}) \cong H^4(W;\mathbb{Z}) \cong H_1(W_f;\partial W_f;A)$ are of rank $l-1$ by the second sequence together with Poincar\'e duality theorem for $W_f$ and Proposition \ref{prop:4} (\ref{prop:4.2}). $H_4(W_f;\mathbb{Z})$ is isomorphic to $H^1(W_f;\partial W_f;\mathbb{Z})$ and $H_1(W_f;\partial W_f;\mathbb{Z})$, free and of rank $l-1$ by Poincar\'e duality theorem for $W_f$. Furthermore,
$H_1(W_f,\partial W_f;\mathbb{Z})$ has a basis each homology class of which is represented a $1$-dimensional connected manifold $L_j$ diffeomorphic to a closed interval enjoying the following properties.
\begin{itemize}
	\item $L_j$ is smoothly embedded in $W_f$.
	\item The boundary $\partial L_j$ consists of exactly two points and is in the boundary $\partial W_f$. The interior ${\rm Int}\ L_j$ is in the interior ${\rm Int}\ W_f$.
	\item ${q_f}^{-1}(L_j)$ can be seen as the $2$-dimensional sphere $S_{L_j}$ of the domain of a Morse function with exactly two singular points.
	\item There exists a connected component $C_0$ of $\partial W_f$ and it contains exactly one point in the boundary $\partial L_j$ for every $j$.
 	\item For each of any pair of distinct manifolds $L_{j_1}$ and $L_{j_2}$, choose a suitable one point in the boundary. Then they are in distinct connected components of $\partial W_f$ which are not the connected component $C_0$ before.
\end{itemize}
By fundamental arguments on Poincare duality theorem or so-called intersection theory, we have the following facts.
\begin{itemize}
	\item $H_4(M;\mathbb{Z})$ is of rank $l-1$.
	\item $H_4(M;\mathbb{Z})$ is free since $M$ is closed and simply-connected. More precisely, $H^2(M;\mathbb{Z})$ must be free and $H_4(M;\mathbb{Z})$ is isomorphic to it by Poincar\'e duality theorem.
			\item $H_4(M;\mathbb{Z})$ has a suitable basis each of which is represented by some connected component of the singular set of $f$, mapped to $\partial W_f$ by the restriction of $f$ there. In addition, remember that $q_f$ maps the singular set by a diffeomorphism onto $\partial W_f$ by Proposition \ref{prop:2} (\ref{prop:2.3}).
				\item For each element of the basis of $H_4(M;\mathbb{Z})$ before, the Poincar\'e dual to it can be regarded as an element represented by each sphere $S_{L_j} \in M$ before.
			\end{itemize}
		
	   By fundamental arguments in \cite{saeki} or easy observations, the singular set of the map $f$ is a $4$-dimensional closed and orientable manifold and its normal bundle (in $M$) is regarded as $2$-dimensional linear bundle and a bundle equivalent to the restriction of a boundary linear bundle and as a result it is trivial. We consider the Poincar\'e duals to homology classes represented by spheres in $\{S_{L_j}\} \subset M$. These Poincar\'e duals are represented by connected components of the singular set of $f$. We perturb a submanifold here or each connected component of the singular set so that the resulting submanifold and the original submanifold are distinct. This means that the cup product of any two elements of $H^2(M;\mathbb{Z})$ is the zero element. Such arguments have been in \cite{kitazawa4,kitazawa5} for example and will be used in our proof of Theorem \ref{thm:5} for example. This completes the proof (\ref{thm:4.3}).
		
		 \end{proof}
\begin{proof}[A proof of Main Theorem \ref{mthm:2}]
Suppose that a $6$-dimensional closed and simply-connected manifold $M$ admits a special generic map $f:M \rightarrow {\mathbb{R}}^5$ such that $H_2(W_f;\mathbb{Z})$ is trivial.
	Suppose also that $H_2(M;\mathbb{Z})$ is not free. Then there exists an element of a finite order of $H_2(M;\mathbb{Z})$ which is not the zero element. 
	Every connected component of the singular set of $f$ is mapped onto some connected component of the boundary of $\partial W_f$ by $q_f$ and a diffeomorphism onto the connected component. $H_4(W_f;\mathbb{Z})$ is free by the argument in the proof of Theorem \ref{thm:4} for example. 
	By a fundamental argument on intersection theory, this element is represented by a $2$-dimensional sphere smoothly embedded in $M$ and we can regard that this and any  connected component of the singular set of $f$, which is mapped onto some connected component of the boundary $\partial W_f \subset W_f$ by $q_f$ and a diffeomorphism onto the connected component, do not intersect.  Thus $H_2(M;\mathbb{Z})$ is free.
	
	From Theorem \ref{thm:4}, if a $6$-dimensional closed and simply-connected manifold $M$ admits a special generic map $f:M \rightarrow {\mathbb{R}}^5$ such that $H_2(W_f;\mathbb{Z})$ is trivial, then $M$ must be a manifold in Main Theorem \ref{mthm:1} by virtue of classifications of $6$-dimensional closed and simply-connected manifolds as before. In addition, by the fact that the 2nd Stiefel-Whitney class of $M$ is the zero element, the linear bundles over $S^2$ must be trivial and the total spaces must be diffeomorphic to $S^2 \times S^4$.

	 It suffices to construct a special generic map into ${\mathbb{R}}^5$ on a desired $6$-dimensional manifold.
	
	We can prepare a compact and simply-connected smooth manifold represented as a boundary connected sum of copies of $S^3 \times D^2$ or $S^4 \times D^1$ and embed the manifold smoothly into ${\mathbb{R}}^5$. Applying Proposition \ref{prop:3} by taking the internal smooth bundle over the resulting manifold $W_f$ and the boundary linear bundle over $W_f$ as trivial bundles and checking the homology groups of the $5$-dimensional manifold $W_f$ and the resulting $6$-dimensional closed and simply-connected manifold $M$ complete the proof.
	
	We give another precise exposition. We consider the product map of a Morse function with exactly two singular points on $S^2$ and the identity map on $S^4$ and the product map of a canonical projection of the unit sphere $S^3$ into ${\mathbb{R}}^2$ and the identity map on $S^3$. By embedding the spaces of the targets suitably, we can regard these two naturally as special generic maps into ${\mathbb{R}}^5$ whose restrictions to the singular sets are embeddings. We can construct a special generic map on a general desired manifold here, represented as a connected sum of these manifolds, easily. Expositions on the construction have been presented in the end of the proof of Main Theorem \ref{mthm:1} for example.

	\end{proof}

\begin{Thm}[\cite{kitazawa5} ($A:=\mathbb{Z}$)]
	\label{thm:5}
	Let $p$ be a power of a prime $p_0$ and represented by $p={p_0}^l$ for some positive integer $l$.
	Let $A$ be the ring $\mathbb{Z}$ of all integers or the commutative ring of the form $\mathbb{Z}/p\mathbb{Z}$, which is of order $p={p_0}^l$ and represented as the natural quotient ring of order $p$. In a word, equivalently, the ring is isomorphic to $\mathbb{Z}$ or a finite field. 	
	If a closed and simply-connected manifold $M$ of dimension $m=6$ admits a special generic map $f:M \rightarrow {\mathbb{R}}^5$ whose singular set is connected, then the cup product $c_1 \cup c_2$ is zero for any pair $c_1,c_2 \in H^2(M;A)$ of cohomology classes.	
\end{Thm}
\begin{proof}[A proof of Theorem \ref{thm:5}]	
    We abuse the notation similarly. 
    $M$ and $W_f$ are simply-connected from the assumption and Theorem \ref{thm:2} (\ref{thm:2.1}).
    
    Let $A$ be an arbitrary commutative ring.
    We show an exact sequence 
    $$\xymatrix{ H_3(W;A) \ar[r]& H_3(W,M;A) \cong H^4(W;A) \cong H^4(W_f;A) \ar[r]^{{\partial}_{\ast}}& \\
    	H_2(M;A) \ar[r]^{{i_M}_{\ast}} & H_2(W;A) \cong H_2(W_f;A) \ar[r] & H_2(W,M;A) \cong H^5(W;A) \cong H^5(W_f;A) \cong \{0\} \ar[r]&}
    $$
    	and
    $$\xymatrix{ H_2(\partial W_f;A) \ar[r]& H_2(W_f,\partial W_f;A) \ar[r]^{{\partial}_{\ast}}& \\
    	H_1(\partial W_f;A) \ar[r] & H_1(W_f;A) \cong \{0\} \ar[r] & H_1(W_f,\partial W_f;A) \cong H^4(W_f;A) \ar[r]& \\
    	H_0(\partial W_f;A) \cong A \ar[r]& H_0(W_f;A) \cong A \ar[r]& \{0\}&}
    $$
    in our proof of Theorem \ref{thm:4} again where we apply Poincar\'e duality theorem for compact and simply-connected manifolds $W$ and $W_f$ for several isomorphisms of groups (modules over $A$) and Proposition \ref{prop:4} (\ref{prop:4.2}) for example.
        
	The last homomorphism from $H_0(\partial W_f;A)$ into $H_0(W_f;A)$ induced by the inclusion of $\partial W_f$ into $W_f$ is an isomorphism since $\partial W_f$ is connected by Proposition \ref{prop:2} (\ref{prop:2.3}). We have $H_1(W_f,\partial W_f;A) \cong H^4(W_f;A) \cong \{0\}$ by this exact sequence and Poincar\'e duality theorem for the compact, connected and orientable manifold $W_f$. We can see that 
	${i_{M}}_{\ast}:H_2(M;A) \rightarrow H_2(W;A)$ is an isomorphism from the first sequence together with Poincar\'e duality theorem for the compact, connected and orientable manifold $W$ and Proposition \ref{prop:4} (\ref{prop:4.2}).
	
	We set $A:=\mathbb{Z}/q\mathbb{Z}$, which is of order $q$, an arbitrary number represented as the power ${q_0}^{k}$ of some prime number $q_0$ and some positive integer $k>0$. We have a basis ${\mathcal{B}}_{M,q}$ of $H_2(M;\mathbb{Z}/q\mathbb{Z})$. Each element of ${\mathcal{B}}_{M,q} \subset H_2(M;\mathbb{Z}/q\mathbb{Z})$ is represented by a smoothly embedded $2$-dimensional sphere in $M$ mapped to $W_f$ enjoying the following two.
	\begin{itemize}
		\item The restriction of $q_f$ to the $2$-dimensional sphere is a smooth embedding.
		\item The intersection of the image of the $2$-dimensional sphere and $\partial W_f$ consists of finitely many points.
	\end{itemize}
	 
	 For each element $e_j \in {\mathcal{B}}_{M,q}$, we have the dual ${{q_f}_{\ast}(e_j)}^{\ast}$ to ${q_f}_{\ast}(e_j)$ as an element of $H^2(W_f;\mathbb{Z}/q\mathbb{Z})$ and its Poincar\'e dual ${\rm PD}({{q_f}_{\ast}(e_j)}^{\ast})$ as an element of $H_3(W_f,\partial W_f;\mathbb{Z}/q\mathbb{Z})$ (where $M$ and $W_f$ are oriented suitably). Note that ${q_f}_{\ast}:H_2(M;\mathbb{Z}/q\mathbb{Z}) \rightarrow H_2(W_f;\mathbb{Z}/q\mathbb{Z})$ is an isomorphism by virtue of the relation $q_f=r_f \circ i_M$ of Proposition \ref{prop:4} (\ref{prop:4.4}) and the fact that the ${i_{M}}_{\ast}:H_2(M;\mathbb{Z}/q\mathbb{Z}) \rightarrow H_2(W;\mathbb{Z}/q\mathbb{Z})$ and ${r_f}_{\ast}:H_2(W;\mathbb{Z}/q\mathbb{Z}) \rightarrow H_2(W_f;\mathbb{Z}/q\mathbb{Z})$ are isomorphisms. We have a basis of $H_2(W_f;\mathbb{Z}/q\mathbb{Z})$ canonically. As we do in the proof of Main Theorem 1 of \cite{kitazawa4}, we canonically have the Poincar\'e dual ${\rm PD}({e_j}^{\ast}) \in H_4(M;\mathbb{Z}/q\mathbb{Z})$ to the dual ${e_j}^{\ast} \in H^2(M;\mathbb{Z}/q\mathbb{Z})$ to $e_j$. More precisely, according to Proposition \ref{prop:4} (\ref{prop:4.4}), $r_f$ is the projection of a bundle over $W_f$ whose fiber is diffeomorphic to $D^{2}$ and we consider a kind of {\rm prism operators} or arguments on so-called {\it Thom classes} to obtain an element of $H_4(W,M;\mathbb{Z}/q\mathbb{Z})$ and take the value of the boundary homomorphism there to obtain a desired element.  
	 Here we consider the intersection theory for $M$ and $W_f$ as we do in the proof of this theorem of \cite{kitazawa4,kitazawa5}. $3$ is the degree of ${\rm PD}({{q_f}_{\ast}(e_j)}^{\ast})$, $5$ is the dimension of $W_f$, and we have $3+3-5=1$. From this, we may regard that the Poincar\'e dual to the cup product of two duals to elements of ${\mathcal{B}}_{M,q}$ is a sum of homology classes represented by the preimages of circles in ${\rm Int}\ W_f$ or $1$-dimensional compact and connected manifolds diffeomorphic to a closed interval in $W_f$. Moreover, for these $1$-dimensional compact and connected manifolds diffeomorphic to a closed interval, we may assume the following properties. Note that a similar situation appears in the proof of Theorem \ref{thm:4}.
	 \begin{itemize}
	 	\item The boundaries are embedded in $\partial W_f$.
	 	\item The interiors are embedded in ${\rm Int}\ W_f$.
	 	\item For the map $q_f$, the preimages are $2$-dimensional spheres where they are regarded as the manifolds of the domains of Morse functions with exactly two singular points.  
	\end{itemize}
	
	The preimages of the circles in ${\rm Int}\ W_f$ are regarded as the boundaries of the total spaces of trivial bundles over a copy of $D^2$ in ${\rm Int}\ W_f$ whose fibers are circles, given by $q_f$. Remember that $W_f$ is simply-connected and that its dimension is $5$ and sufficiently high.
	
	We investigate each $2$-dimensional sphere ${S^2}_{j^{\prime}}$, represented as the preimage of a closed interval in $W_f$ before.
    
	$W_f$ is simply-connected and $\partial W_f$ is connected. This implies the existence of a (smooth) homotopy $H_{f,j^{\prime}}:{S^2}_{j^{\prime}} \times [0,1] \rightarrow W_f$ from the composition of the original embedding of each $2$-dimensional sphere ${S^2}_{j^{\prime}}$ here into $M$ with $q_f$ to a constant map whose image is a one-point set in $\partial W_f$. For $0 \leq t \leq 1$, we can define $S_{{\rm H},f,j^{\prime}}(t)$ as the set of all points in ${S^2}_{j^{\prime}}$ where the pairs of the points and $t$ are mapped into $\partial W_f$ by this homotopy. We may assume the property $S_{{\rm H},f,j^{\prime}}(t_1) \subset S_{{\rm H},f,j^{\prime}}(t_2)$ for any $0<t_1<t_2<1$. The original embedding of the $2$-dimensional sphere ${S^2}_{j^{\prime}}$ into $M$ is shown to be (smoothly) null-homotopic. This completes the proof of Theorem \ref{thm:3} for $A:=\mathbb{Z}/q\mathbb{Z}$.
	
	Let $A:=\mathbb{Z}$. $H_2(M;\mathbb{Z})$ and $H_2(W_f;\mathbb{Z})$ are isomorphic as commutative groups as in the case where $A$ is a finite field before. They are isomorphic to the direct sum of a free commutative group $G_{\rm Free}$ and a finite commutative group $G_{\rm Finite}$. $G_{\rm Finite}$ is isomorphic to a direct sum of finitely many cyclic groups each of which is of order $p_{j^{\prime \prime}}={{p_{j^{\prime \prime}}}_0}^{l_{j^{\prime \prime}}}$ for a suitable prime ${p_{j^{\prime \prime}}}_0$ and a suitable integer $l_{j^{\prime \prime}}>0$. This is due to a fundamental classification theorem for finitely generated commutative groups. We (can) abuse $G_{\rm Free}$ and $G_{\rm Finite}$ for summands of suitable decompositions into internal direct sums of $H_2(M;\mathbb{Z})$ and $H_2(W_f;\mathbb{Z})$. We can define a similar basis ${\mathcal{B}}_M:=\{e_j\}_j \subset G_{\rm Free} \subset H_2(M;\mathbb{Z})$ for the summand $G_{\rm Free} \subset H_2(M;\mathbb{Z})$ and we can argue similarly. This is also similar to a main argument in the proof of Main Theorem \ref{mthm:1}. We have the Poincar\'e dual ${\rm PD}({e_j}^{\ast}) \in H_4(M;\mathbb{Z})$ as before where we abuse the previous notation. As in the proof of Main Theorem \ref{mthm:1}, this is represented by a $4$-dimensional, closed, connected and oriented manifold regarded as the manifold of the domain of a special generic map into a $3$-dimensional non-closed and orientable manifold with no boundary. This comes from the fact that 
    each homology class of $H_3(W_f,\partial W_f;\mathbb{Z})$ is represented by a $3$-dimensional compact, connected and oriented smooth manifold. This is also a very explicit case of \cite{thom}. We can argue as in the case where $A$ is a finite field to complete the proof of Theorem \ref{thm:3}. This also reviews some ingredients of the proof in \cite{kitazawa5}.

We can also prove Theorem \ref{thm:5} in the case where the dimension $m$ is greater than $6$. This is in \cite{kitazawa4}.
\end{proof}
\begin{proof}[A proof of Main Theorem \ref{mthm:3}]
	We abuse the notation similarly.
	According to our proof of Theorem \ref{thm:5}, the existence of a special generic map $f:M \rightarrow {\mathbb{R}}^5$ here yields the isomorphism $q_f:H_2(M;\mathbb{Z}) \rightarrow H_2(W_f;\mathbb{Z})$ together with Proposition \ref{prop:4} (\ref{prop:4.4}) and the relation $q_f=r_f \circ i_M$ for the map $r_f:W \rightarrow W_f$ giving the collapsing.
	
	The 2nd Stiefel-Whitney class of $M$ is an element of $H^2(M;\mathbb{Z}/2\mathbb{Z})$ and Theorem \ref{thm:4} implies that it is zero. Furthermore, the 1st Pontrjagin class of $M$, which is oriented in an arbitrary way, is an element in $H^4(M;\mathbb{Z})$ and zero by Theorem \ref{thm:4}. The topology and the differentiable structure of this $6$-dimensional, closed and simply-connected manifold $M$ are determined by an element $\gamma \in H^4(M;\mathbb{Z})$. Furthermore, $4\gamma$ and the 1st Pontrjagin class of (an arbitrary oriented) $M$ agree. This is due to general theory of classifications of $6$-dimensional, closed and simply-connected manifolds. Consult \cite{jupp,wall,zhubr,zhubr2} again. By the assumption on elements of finite orders of the homology group $H_2(M;\mathbb{Z})$, $\gamma$ must be zero. Note that $H^4(M;\mathbb{Z})$ is isomorphic to $H_2(M;\mathbb{Z})$ by virtue of Poincar\'e duality theorem.
	
	To complete the proof, it is sufficient to construct a special generic map $f:M \rightarrow {\mathbb{R}}^5$ on a $6$-dimensional closed and simply-connected manifold $M$ such that $H_2(M;\mathbb{Z})$ is isomorphic to an arbitrary finite commutative group $G$ and that $H_3(M;\mathbb{Z})$ is of rank $2l$ for an arbitrary non-negative integer $l \geq 0$. Note that $H_3(M;\mathbb{Z})$ is isomorphic to the direct sum of a free commutative group of rank $2l$ and a group isomorphic to $G$ by virtue of universal coefficient theorem and Poincar\'e duality theorem. 
	
 We show an exact sequence as in the proof of Theorem \ref{thm:5}
$$\xymatrix{ H_4(W;\mathbb{Z}) \cong H_4(W_f;\mathbb{Z}) \cong H^1(W_f;\partial W_f;\mathbb{Z}) \cong \{0\} \ar[r]&\\ H_4(W,M;A) \cong H^3(W;A) \cong H^3(W_f;A) \ar[r]^{{\partial}_{\ast}}& \\
	H_3(M;\mathbb{Z}) \ar[r]^{{i_M}_{\ast}} &\\ H_3(W;\mathbb{Z}) \cong H_3(W_f;\mathbb{Z}) \ar[r] &\\ H_3(W,M;\mathbb{Z}) \cong H^4(W;\mathbb{Z}) \cong H^4(W_f;\mathbb{Z}) \cong H_1(W_f,\partial W_f;\mathbb{Z}) \cong \{0\} \ar[r]&}
$$

where we abuse $W$ and apply some important propositions and theorems as before. 
The rank of $H_3(M;\mathbb{Z})$ is shown to be twice the rank of $H_3(W_f;\mathbb{Z})$.

	By a fundamental argument on handles in the smooth category, we easily have a $5$-dimensional, compact and simply-connected manifold $W_f$ in Proposition \ref{prop:3} satisfying the following two.
	\begin{itemize}
		\item The boundary $\partial W_f$ is not empty and it is connected.
		\item $H_2(W_f;\mathbb{Z})$ is isomorphic to $G$.
		\item $H_3(W_f;\mathbb{Z})$ is of rank $l$.
		
		\item $W_f$ collapses to a $3$-dimensional polyhedron and has the (simple) homotopy type of a $3$-dimensional polyhedron.
	
		\item The tangent bundle of $W_f$ is trivial and by well-known studies of smooth immersions by Hirsh for example, we can smoothly immerse $W_f$ into ${\mathbb{R}}^5$.
	\end{itemize}

The classification theory tells us that the rank of the $3$rd homology group of a $6$-dimensional closed and simply-connected topological manifold must be even where the coefficient ring is $\mathbb{Z}$. In addition, it also tells that the manifold is represented as a connected sum of another $6$-dimensional closed and simply-connected manifold whose 2nd homology group is finite and finitely many copies of $S^3 \times S^3$ where the coefficient ring is $\mathbb{Z}$. Furthermore, more rigorously, the connected sum here is considered in the topology category if the manifold is a topological manifold, in the PL category if the manifold is a PL manifold, and in the smooth category if the manifold is a smooth manifold.

By applying Proposition \ref{prop:3} where the internal smooth bundle and the boundary linear bundle over $W_f$ are trivial, we have a desired special generic map on a desired $6$-dimensional closed and simply-connected manifold $M$.

This completes our proof.
\end{proof}

For other explicit construction of a special generic map for Main Theorems \ref{mthm:2} and \ref{mthm:3} and on a $6$-dimensional closed and simply-connected manifold $M$ such that $H_j(M;\mathbb{Z})$ is not free, apply Proposition \ref{prop:3} with some $5$-dimensional compact and connected manifolds whose boundaries are connected and which we can smoothly immerse into ${\mathbb{R}}^5$ presented in \cite{kitazawa2,kitazawa3} or ones we can easily obtain from these examples.

\begin{proof}[A proof of Main Theorem \ref{mthm:4}]
	As the proofs of other (Main) Theorems, we abuse the notation.

	We prove (\ref{mthm:4.1}). For $m=5,6$, this is already shown as Theorem \ref{thm:1} (\ref{thm:1.3}) and Main Theorem \ref{mthm:1}. Let $m \geq 7$. As in the proof of Main Theorem \ref{mthm:1}, we can know that $W_f$ is a $4$-dimensional closed and simply-connected manifold smoothly immersed into ${\mathbb{R}}^4$. We can also know that ${q_f}_{\ast}:H_2(M;\mathbb{Z}) \rightarrow H_2(W_f;\mathbb{Z})$ is an isomorphism and these groups are free. Theorem \ref{thm:2} (\ref{thm:2.1}) also implies that ${q_f}_{\ast}:H_j(M;\mathbb{Z}) \rightarrow H_j(W_f;\mathbb{Z})$ is an isomorphism for $0 \leq j \leq m-4$.
	$W_f$ has the homotopy type of a $3$-dimensional polyhedron.
	Thus $H_j(M;\mathbb{Z})$ is the trivial group if $j \neq 0,2,3,m-3,m-2,m$ where we apply Poincar\'e duality theorem. It also follows that $H_3(M;\mathbb{Z})$ is free. 
	We can take a basis of $H_{j_0}(W_f;\mathbb{Z})$ for $j_0=2,3$. Each element of the basis is represented by a $j_0$-dimensional, closed, connected and oriented manifold smoothly embedded in ${\rm Int}\ W_f$. The Poincar\'e dual to the element is represented by a ($4-j_0$)-dimensional, compact, connected and oriented manifold smoothly embedded in $W_f$ satisfying the following condition: the boundary is embedded in the boundary $\partial W_f$ and the interior is embedded in the interior ${\rm Int}\ W_f$. The preimage ${q_f}^{-1}(Y)$ of the ($4-j$)-dimensional manifold $Y$ is regarded as the domain of a special generic map into a ($4-j_0$)-dimensional, non-closed and orientable manifold with no boundary by the definition and local properties on the structure of a special generic map. This is due to the related similar exposition in the proof of Main Theorem \ref{mthm:1}. As a result the special generic map is regarded as a one into ${\mathbb{R}}^{4-j_0}$ by Propositions \ref{prop:2} and \ref{prop:3}. The (oriented) manifold ${q_f}^{-1}(Y)$ bounds an oriented compact manifold $W$ by Proposition \ref{prop:4} (\ref{prop:4.1}). The manifold $W$ may not be smooth. However, it has no problem. The ($m-j_0$)-th homology class represented by this manifold is regarded as the Poincar\'e dual to the class $({q_f}^{\ast})^{-1}(e_Y)$ where $e_Y \in H_{j_0}(W_f;\mathbb{Z})$ denotes the original element of the basis. 
	
	$Y$ is embedded smoothly in $W_f$ with a trivial normal bundle by a explicit fundamental argument on linear bundles. As a result, the manifold ${q_f}^{-1}(Y)$ is, as a submanifold, embedded smoothly in $M$ with a trivial normal bundle. In fact, we need sophisticated theory of local properties of special generic maps here and more general smooth maps due to Thom for example: so-called Thom's isotopy theorem. ${q_f}^{-1}(Y)$ is a homotopy sphere or a manifold in Theorem \ref{thm:1} (\ref{thm:1.1}) by Reeb's theorem and Theorem \ref{thm:1} (\ref{thm:1.1}). Note also that the set of all $({q_f}^{\ast})^{-1}(e_Y)$ here obtained by considering every $e_Y$ is regarded as a basis of the group $H_{m-j_0}(M;\mathbb{Z})$.
	
	Thus by several fundamental arguments on linear bundles, the ($m-j_0$)-th Stiefel Whitney class of $M$ is the zero class. In the case where $m-j_0$ is divisible by $4$, for any oriented $M$, the ($\frac{m-j_0}{4}$)-th Pontrjagin class of $M$ is the zero class.

	Remember that a similar and more explicit fact has been proved in the proof of Main Theorem \ref{mthm:1} and that we adopt an exposition which is a bit different from the original one.	
	
		The 3rd Stiefel-Whitney class of $M$ is the zero element since $H_2(M;\mathbb{Z})$ is free. This is due to general theory of linear bundles. $H_j(M;\mathbb{Z})$ has been shown to be the trivial groups if $j \neq 0,2,3,m-3,m-2,m$. Corollary 3.18 of \cite{saeki} says that the $j$-th Stiefel-Whitney class of $M$ is the zero element for $j > m-4+1=m-3$ and that the $j$-th Pontrjagin class is the zero element for $j > \frac{m-4+1}{2}$ or $4j > 2(m-3)>m-3$. This completes the proof of (\ref{mthm:4.1}).
		
		We prove (\ref{mthm:4.2}). We construct special generic maps on $M^{7,0}$ into ${\mathbb{R}}^n$ for $n=4,5$. We can prepare a smoothly embedded copy of $S^2 \times D^{n-2} \subset {\mathbb{R}}^n$ and that of $S^3 \times D^{n-3}$ easily as we do in several proofs in the present paper. We apply construction as in the end of the proof of Main Theorem \ref{mthm:1} to obtain desired special generic maps for example.
		
		It follows from well-known explicit theory on linear bundles that there exists a linear bundle over $S^4$ whose fiber is the $3$-dimensional unit sphere and whose total space $M$ is a $7$-dimensional closed and simply-connected oriented manifold enjoying the following properties.
		\begin{itemize}
			\item $M$ is simply-connected and $H_j(M;\mathbb{Z})$ is isomorphic to $\mathbb{Z}$ for $j=0,3,4,7$ and the trivial group for $j=1,2,5,6$.
			\item The $1$st Pontrjagin class, defined uniquely as an element of $H^4(M;\mathbb{Z})$, is $4k$ times a generator of $H^4(M;\mathbb{Z})$.
			\end{itemize}
			By the structure of the bundle before and Proposition \ref{prop:3}, we have a special generic map $f$ over the $5$-dimensional manifold $W_f$, diffeomorphic to $S^4 \times D^1$.

		As a fundamental genaral fact, we note that reversing the orientation of a given oriented manifold changes the sign of the $1$st Pontrjagin class of the manifold.
			
			To obtain a desired family $\{M^{7,\lambda}\}_{\lambda \in \Lambda}$ and special generic maps on these manifolds into ${\mathbb{R}}^5$, we apply a similar method of construction. In other words, we use this new map and the previously presented maps into ${\mathbb{R}}^5$ whose images are copies of $S^2 \times D^3$ and $S^3 \times D^2$ smoothly embedded in ${\mathbb{R}}^5$. We consider connected sums of the $7$-dimensional manifolds in general.
			
			This completes the proof.

\end{proof}

\section{Final remarks.}
\begin{Rem}
	\label{rem:1}
Related to Theorem \ref{thm:1} (\ref{thm:1.2}), we can construct special generic maps into ${\mathbb{R}}^3$ on manifolds represented as connected sums of the total spaces of linear bundles over $S^2$ whose fibers are diffeomorphic to $S^k$ for $k \geq 2$. We can know this through the original paper. This is also an exercise. Nishioka's construction in the proof of Main Theorem \ref{mthm:1} is also regarded as a higher dimensional version.

\cite{saeki} also shows that a closed and simply-connected manifold whose dimension is greater than $3$ admitting a special generic map into ${\mathbb{R}}^3$ must be represented as a connected sum of the total spaces of smooth bundles over $S^2$ whose fibers are either of the following two (where the connected sum is considered in the smooth category). We have encountered arguments of this type in the present paper for several times.
\begin{itemize}
	\item A homotopy sphere whose dimension is greater than $1$ and not $4$.
	\item A $4$-dimensional standard sphere.
\end{itemize}
Note also that $W_f$ in Propositions \ref{prop:2} and \ref{prop:3} for a special generic map $f$ must be represented as a boundary connected sum of finitely many copies of $S^2 \times D^1$. Note also that in \cite{saeki}, so-called Poincare's conjecture for $3$-dimensional spheres was regarded as unsolved and that we use the affirmative answer here.  
\end{Rem}
\begin{Rem}
	\label{rem:2}
	Besides Main Theorem \ref{mthm:4}, we do not know variants of our Main Theorems for closed and simply-connected manifolds whose dimensions are greater than $6$. It seems to be difficult to find explicit classifications of closed and simply-connected manifolds of certain classes in general. Although there exist such classifications, it seems to be difficult to apply them. \cite{kreck} is a result for $7$-dimensional ones whose 2nd homology groups are free where the coefficient ring is $\mathbb{Z}$. We do not know how to use this to obtain similar results or variants of Main Theorems \ref{mthm:1}, \ref{mthm:2} and \ref{mthm:3}.
	
	It is well-known that the differentiable structures of homotopy spheres admitting special generic maps into Euclidean spaces whose dimensions are sufficiently high and lower than the dimensions of the homotopy spheres are strongly restricted. For this see \cite{calabi}, (section 4 of) \cite{saeki}, \cite{saeki2} and \cite{wrazidlo} for example. Theorem \ref{thm:1} (\ref{thm:1.2}) is also regarded as a related fact in the case where the dimensions of the manifolds of the domains are $m=4$. For related examples of $4$-dimensional manifolds which are homeomorphic to and not diffeomorphic to the manifolds, see also \cite{saekisakuma2} for example.
\end{Rem}

	Related to some of Remark \ref{rem:2} and our main study, \cite{kitazawa4,kitazawa5} mainly concern some other restrictions on cohomology rings of closed and simply-connected manifolds of dimensions at least $6$. Some arguments in the present study are due to them.

\section{Acknowledgement.}
\label{sec:4}
The author is a member of the project JSPS KAKENHI Grant Number JP17H06128 "Innovative research of geometric topology and singularities of differentiable mappings": Principal investigator is Osamu Saeki. The present study is supported by this project. We also declare that data supporting the present study essentially are all in the present paper.

\end{document}